\documentclass[10pt]{article}

\font\smallit=cmti10

\usepackage{amssymb,latexsym,amsmath,amsthm,amsfonts,epsfig}
\usepackage[showonlyrefs]{mathtools} 
\usepackage{xcolor}
\usepackage{listings}
\usepackage{url}
\usepackage[numbers,sort&compress]{natbib}
\setcitestyle{square,comma}

\makeatletter
\renewcommand\section{\@startsection {section}{1}{\z@}
  {-30pt \@plus -1ex \@minus -.2ex}
  {2.3ex \@plus .2ex}
  {\normalfont\normalsize\bfseries\boldmath}}
\renewcommand\subsection{\@startsection{subsection}{2}{\z@}
  {-3.25ex\@plus -1ex \@minus -.2ex}
  {1.5ex \@plus .2ex}
  {\normalfont\normalsize\bfseries\boldmath}}
\renewcommand{\@seccntformat}[1]{\csname the#1\endcsname. }

\newcommand\blfootnote[1]{%
  \begingroup
  \renewcommand\thefootnote{}\footnote{#1}%
  \addtocounter{footnote}{-1}%
  \endgroup
}
\makeatother

\newtheorem{theorem}{Theorem}
\newtheorem{lemma}{Lemma}

\theoremstyle{definition}

\theoremstyle{remark}
\newtheorem{remark}{Remark}
\newtheorem{example}{Example}

\numberwithin{equation}{section}

\newcommand{\ceil}[1]{\left\lceil #1 \right\rceil}

\definecolor{codegreen}{rgb}{0,0.6,0}
\definecolor{codegray}{rgb}{0.5,0.5,0.5}
\definecolor{codepurple}{rgb}{0.58,0,0.82}
\definecolor{backcolour}{rgb}{0.95,0.95,0.92}
\lstdefinestyle{mystyle}{
  backgroundcolor=\color{backcolour},
  commentstyle=\color{codegreen},
  keywordstyle=\color{magenta},
  numberstyle=\tiny\color{codegray},
  stringstyle=\color{codepurple},
  basicstyle=\ttfamily\footnotesize,
  breaklines=true,
  captionpos=b,
  keepspaces=true,
  numbers=left,
  numbersep=5pt,
  showspaces=false,
  showstringspaces=false,
  showtabs=false,
  tabsize=4
}
\begin{document}

\begin{center}
\uppercase{\bf \boldmath An Algorithm for the Product-Sum Equality}
\vskip 20pt
{\bf Hlib Husarov}\\
{\smallit Mary Immaculate Secondary School, Lisdoonvarna, Ireland}\\
\vskip 10pt
{\bf Eberhard Mayerhofer}\\
{\smallit Mathematics and Statistics Department, University of Limerick, Ireland}\\
{\tt eberhard.mayerhofer@ul.ie}\\
\end{center}

\vskip 20pt
\centerline{\smallit Received: 13 August 2025}
\vskip 30pt

\centerline{\bf Abstract}
\noindent
We propose a recursive algorithm for identifying all finite sequences of positive integers whose product equals their sum. Our method uses solutions of strictly shorter length that are iteratively extended in pursuit of a valid solution. The algorithm is efficient, with a time complexity similar to the quick sort algorithm.

\pagestyle{myheadings}
\thispagestyle{empty}
\baselineskip=12.875pt
\vskip 30pt


\section{Introduction}\label{sec: intro}
This paper puts forward a recursive algorithm to identify all positive integer sequences $a=(a_1,a_2,\dots,a_n)$ that satisfy the Diophantine equation
\begin{equation}\label{eq: sumprod}
\prod_{i=1}^n a_i = \sum_{i=1}^n a_i
\end{equation}
and gives a precise asymptotic result concerning its time-complexity.

While this problem has been studied with the help of computer programs---especially in the search for a number~$n$ with at most one solution---the underlying algorithms are rarely disclosed, and concrete implementations appear to be unavailable in the public domain.\footnote{Only while finalizing this manuscript did we come across \cite{nyblom2013}, which presents another recursive method that constructs solutions to~\eqref{eq: sumprod} for any integer~$n$ by generating solutions for all $n' < n$, combined with certain divisibility checks. Our method is conceptually different: it generates those solutions for $n'<n$ as a byproduct without the need to store excessive amounts of data, and divisibility checks are unnecessary.}

The main results known to date regarding this problem are as follows.
\begin{enumerate}
\item For each $n$, the problem has at least one solution, called the basic solution, given by
\begin{equation}\label{eq: basic}
a_1 = n, \quad a_2 = 2, \quad a_3 = \dots = a_n = 1.
\end{equation}
An integer $n$ for which the basic solution is the only one is called \emph{exceptional}.
\item The number of solutions grows on average faster than any logarithmic power: Let $f(n)$ denote the number of all solutions of length $n$ satisfying \eqref{eq: sumprod}. Then the average satisfies (in a simplified form; the original source \cite{weingartner} reports more refined asymptotics):
\[
\frac{1}{n} \sum_{n \leq x} f(n) = O\left(\frac{e^{2\sqrt{\log(x)}}}{\log^{3/4}(x)}\right).
\]

\item \cite{ecker,nyblom,reble} If $n$ is exceptional, then $n-1$ is a Sophie Germain prime, meaning both $n-1$ and $2n-1$ are prime.\footnote{Nyblom provides a detailed proof of this fact, while Reble offers an explicit construction: if $2n-1 = ab$ is composite with $a\geq b \geq 2$, then a solution of length $n$ is given by $a_1 = (a+1)/2$, $a_2 = (b+1)/2$, $a_3=2$ and $a_4 = \dots = a_n = 1$.}

\item \cite{ecker,reble} If $n$ is exceptional, then $n \equiv 0 \pmod{30}$ or $n \equiv 24 \pmod{30}$.\footnote{While \cite{ecker} does not exclude the possibility $n \equiv 12 \pmod{30}$, \cite{reble} constructs explicit solutions for these cases; for $n = 30j + 12$, the solution is $a_1 = \dots = a_4 = 2$, $a_5 = 2j + 1$, and $a_6 = \dots = a_n = 1$.}

\item \cite{weingartner} The only known exceptional numbers less than $10^{11}$ are
\[
2, 3, 4, 6, 24, 114, 174, 444.
\]
(this is sequence A033179 in the Online Encyclopedia of Integer Sequences, cf.~\url{https://oeis.org/A033179}.)

\item \cite{ecker} If $a = (a_i)_{i=1}^n$ satisfies \eqref{eq: sumprod}, then the following hold:
\begin{enumerate}
\item \label{bta} $a_i \leq n$ for all $1 \leq i \leq n$.
\item \label{btb} $\prod_{i=1}^n a_i \leq 2n$.
\end{enumerate}
Moreover, if $a$ is not the basic solution, then the inequalities in \eqref{bta} and \eqref{btb} are strict.
\end{enumerate}

\subsection{Program of the paper}
We begin by presenting the algorithm in Section~\ref{sec: alg} and perform it for small $n$ in Example \ref{ex: 1}. We prove in Theorem~\ref{th: alg} that it finds all possible solutions. The subsequent statements demonstrate the algorithm's efficiency (Theorem \ref{th: eff working}) and time-complexity (Theorem ~\ref{thm: complexity}).

The proof of Theorem~\ref{th: alg} relies on key results from Section~\ref{sec: inequalities}, in particular a fundamental sharp inequality~\eqref{inequality} for solutions of~\eqref{eq: sumprod}, which bounds the product of the first $k$ elements of a solution $a = (a_1, \dots, a_n)$ in terms of corresponding sums (Theorem~\ref{prop: 1}). Since products grow faster than sums, solutions to~\eqref{eq: sumprod} typically contain many ones; to formalize this intuition, we provide a precise numerical bound on the maximum number of summands different from one (Lemma \ref{lem: log length}). Theorem~\ref{prop: 1} is also used to give an alternative proof of the Boundedness Propositions from~\cite{ecker}, which allows us to search for each unknown $a_i$ in the limited range $\{2,3,\dots,n\}$.


\section{The Algorithm}\label{sec: alg}

Throughout this paper, we assume without loss of generality that the integer sequences $(a_i)_i$ are non-increasing, i.e., $a_1 \geq a_2 \geq \dots \geq a_n$, whether or not they satisfy \eqref{eq: sumprod}. 

Let $n\geq 2$ be a positive integer. The following recursive algorithm searches for solutions to \eqref{eq: sumprod}. It maintains a candidate solution $a = (a_1, \dots, a_i)$, where $i \geq 1$ is the current index. Define the partial sum and partial product of the candidate solution as
\[
s = \sum_{j=1}^i a_j, \quad p = \prod_{j=1}^i a_j,
\]

Due to Theorem \ref{prop: 1}, if $a$ can be extended to a solution of \eqref{eq: sumprod} of length $n$, than it must satisfy the feasibility condition

\begin{equation}\label{eq: feasibility}
p \leq s + n - i, \quad i\leq n;
\end{equation}
Equality holds, whenever the solution is found (in which case $a_{i+1}=\dots a_{n}=1$).

The recursive function takes the triple $(p, s, i)$ that satisfies \eqref{eq: feasibility} and extends the candidate solution $a$ by adding a new element $a_{i+1}$ (we typically start at $p=1$ (empty product), $s=0$ (empty sum) and $i=0$, which also satisfies \eqref{eq: feasibility}. )

The algorithm proceeds as follows:

\begin{enumerate}
  \item If $i \geq 1$, set the maximum allowed value for $a_{i+1}$ as $m := a_i$; otherwise, set $m := n$.

  \item For each candidate extension $a' = 2, 3, \dots, m$:
  \begin{enumerate}
    \item Extend the partial solution by setting $a_{i+1} := a'$.

    \item Update the sum and product:
    \[
    s' := s + a', \quad p' := p \times a'.
    \]

    \item Check the feasibility condition for the extended partial solution:
    \begin{itemize}
      \item If \(p' > s' + n - (i+1)\), prune this branch and backtrack.

      \item If \(p' = s' + n - (i+1)\), record the solution
      \[
      (a_1, \dots, a_i, a_{i+1}, \underbrace{1, \dots, 1}_{n - i - 1 \text{ times}}),
      \]
      then prune this branch and backtrack.

      \item Otherwise, recursively call the function with updated parameters $(p', s', i+1)$ to continue extending the solution.
    \end{itemize}

  \end{enumerate}

  \item After all candidate values $a'$ have been tested, return from the current recursive call.
\end{enumerate}
We explain the algorithm by using a particularly small number in the following:
\begin{example}\label{ex: 1}
For $n=5$, there are three solutions (In the following, we do not quote any ones in a solution. E.g., the basic solution is written as $(5,2)$.). We start
the routine with $p=1,s=0$ and $i=0$.
\begin{enumerate}
  \item Try $a=(2)$. Since $2 < 2 + (5 - 1) = 6$, extend the vector:
  \begin{enumerate}
    \item Try $a=(2,2)$. Since $4 < 4 + (5 - 2) = 7$, extend the vector:
    \begin{enumerate}
      \item Try ${\bf a=(2,2,2)}$. Since $8 = 6 + (5 - 3) = 8$, a solution is found. Return to the last index $i$ where $a_i$ can be increased.
    \end{enumerate}
  \end{enumerate}

  \item Try $a=(3)$. Since $3 < 3 + (5 - 1) = 7$, extend the vector:
  \begin{enumerate}
    \item Try $a=(3,2)$. Since $6 < 5 + (5 - 2) = 8$, extend the vector:
    \begin{enumerate}
      \item Try $a=(3,2,2)$. Since $12 > 7 + (5 - 3) = 9$, the inequality is violated. Return to the last valid index.
    \end{enumerate}
    \item Try ${\bf a=(3,3)}$. Since $9 = 6 + (5 - 2) = 9$, a solution is found. Return to the last index $i$ where $a_i$ can be increased.
  \end{enumerate}

  \item Try $a=(4)$. Since $4 < 4 + (5 - 1) = 8$, extend the vector:
  \begin{enumerate}
    \item Try $a=(4,2)$. Since $8 < 6 + (5 - 2) = 9$, extend the vector:
    \begin{enumerate}
      \item Try $a=(4,2,2)$. Since $16 > 8 + (5 - 3) = 10$, the inequality is violated. Return to the last valid index.
    \end{enumerate}
    \item Try $a=(4,3)$. Since $12 > 7 + (5 - 2) = 10$, the inequality is violated. Backtrack to the last valid index and continue.
  \end{enumerate}

  \item Try $a=(5)$:
  \begin{enumerate}
    \item $(5,2)$ is the basic solution, end.
  \end{enumerate}
\end{enumerate}
\end{example}

In the following three subsections, we develop statements that reveal the efficiency of the algorithm. 

\subsection{Completeness}

The algorithm does not consider every integer sequence of length \( n \); instead, it restricts sequences by imposing \( a_i \leq n \). Moreover, if at any point the product \( p \) satisfies \( p \geq s + n - i \), the algorithm stops extending the sequence \( (a_1, \dots, a_i) \) by adding numbers greater than one. These restrictions do not exclude any solutions:

\begin{theorem}\label{th: alg}
Let $n\geq 3$. Starting from the initial values \((s=0, p=1, i=0)\), the algorithm finds all solutions to equation \eqref{eq: sumprod}. The last solution identified by the algorithm is the basic solution \eqref{eq: basic}.
\end{theorem}
\begin{proof}
To prove the first part, it suffices to show that any candidate solutions omitted by the algorithm do not satisfy the equation.

First, Lemma \ref{lem: 1} guarantees that any valid solution must satisfy \( a_i \leq n \), which justifies restricting the maximum value of \( a_i \) in the loop to \( m \).

Second, by Theorem \ref{prop: 1}, if \( p > s + n - i \), no solution \( a \) to \eqref{eq: sumprod} exists with initial terms \( a_1, \dots, a_i \). If \( p = s + n - i \), then \( (a_1, a_2, \dots, a_i, 1, \dots, 1) \) is a solution. In both cases, extending the sequence further with terms \( a_{i+1} > 1 \) is unnecessary due to Lemma \ref{lem: 2} and Lemma \ref{lem: 2x}, explaining why the corresponding branches terminate in the Algorithm.

For the base case \( i=0 \), we have \( m = n \), so the for-loop runs from \( a = 2 \) to \( a = n \). In particular, when \( a_1 = n \), we get \( s = p = n \), which satisfies \( p < s + (n - 1) \). The recursive call proceeds, immediately assigning \( a_2 = 2 \) and thereby finding the basic solution, then returning to index $i=1$, and since \( a_1 = n=m \), the loop concludes here.
\end{proof}

\subsection{Relation to Nyblom's Algorithm}

Note that \cite{nyblom2013} recursively generates all sets $S_r(n)$ for $2 \leq r \leq m+1$, where $m \leq \log_2(n) + 1$, and $S_r(n)$ denotes the set of all solutions to \eqref{eq: sumprod} with exactly $r$ coefficients satisfying $a_1 \geq a_2 \geq \dots \geq a_r > 1$.
Even in the case $r = 2$, the construction of $S_r(n)$ in \cite{nyblom2013} involves determining the prime factorization of $n - 1$, for which no efficient (that is, of polynomial-time) algorithm is known. Although one could sidestep this issue in the special case $r = 2$ by selecting a larger integer $N > n$ such that $N - 1$ is prime, this workaround becomes impractical for $r \neq 2$ due to the significantly more complex divisibility conditions involved.

In contrast to \cite{nyblom2013}, our algorithm performs no explicit divisibility checks. However, like the method in \cite{nyblom2013}, it generates all solutions for $n' < n$, but does so in a significantly more efficient manner. This claim is proved next (In the following, we adopt a simplified notation to describe the recursive step of the algorithm):
\begin{theorem}\label{th: eff working}
Let $n\geq 3$. The algorithm efficiently generates all solutions of
\begin{equation}\label{eq: sumprod1}
\prod_{j=1}^{n'} a_j = \sum_{j=1}^{n'} a_j
\end{equation}
for any $n'<n$:
\begin{enumerate}
\item Whenever $p'<s'+(n-i')$, then $\widetilde a=(a_1,a_i,a',1,\dots,1)\in\mathbb N^{n'}$ with $n':=p'-s'+i'$ solves 
\eqref{eq: sumprod1}, and $i\leq n'<n$.
\item \label{part: b} The algorithm produces all positive solutions of \eqref{eq: sumprod1} for any $n'< n$.
\item \label{part: c} The algorithm is efficient in that no solution is produced twice.
\end{enumerate}
\end{theorem}
\begin{proof}
The first statement is just a reformulation of Lemma \ref{lem: nyblom} below. Proof of Part \ref{part: c}: Since the algorithm does not produce a finite sequence twice, and since the same sequence for which \eqref{eq: ineq fund} holds cannot satisfy
equality in \eqref{eq: sumprod1} for two different $n'$, we have a one-to-one map between those sequences that satisfy \eqref{eq: ineq fund} for $n$ and the solutions of \eqref{eq: sumprod1} for $n'<n$.
\end{proof}

\begin{lemma}\label{lem: nyblom}
Suppose $a=(a_1,a_2,\dots,a_i)$ with $i\leq n$ and $a_i>1$ solves
\begin{equation}\label{eq: ineq fund}
\prod_{j=1}^i a_j<\sum_{j=1}^i a_j+(n-i),
\end{equation}
then $n':=\prod_{j=1}^i a_j-\sum_{j=1}^i a_j+i$ satisfies $i\leq n'<n$ and  $a=(a_1,a_2,\dots,a_i,1,\dots,1)\in\mathbb N^{n'}$ solves \eqref{eq: sumprod1}.)

Conversely, if for some $n'<n$, $a=(a_1,a_2,\dots,a_i,1,\dots,1)\in\mathbb N^{n'}$ solves \eqref{eq: sumprod1}, then $a$ satisfies \eqref{eq: ineq fund}.
\end{lemma}
\begin{proof}
Note that by definition of $n'$, $a=(a_1,a_2,\dots,a_i,1,\dots,1)$ satisfies \eqref{eq: sumprod1}, and $n'<n$, due to strict inequality in \eqref{eq: ineq fund}.  Furthermore, since $i=n'-(\prod_{j=1}^i a_j-\sum_{j=1}^i a_j)$, and since products of integers strictly greater than one exceed their sum, we have $i\leq n'$.

The second part is trivial, \eqref{eq: sumprod1} implies that $\prod_{j=1}^i a_j=\sum_{j=1}^i a_j+(n'-i)<\sum_{j=1}^i a_j+(n-i)$,
as $n'<n$.
\end{proof}

\subsection{Efficiency}
Finally, we establish the time-complexity of our algorithm. To have a simple measure of time-complexity, we count how many products
$(p,s,i)$ are checked in the algorithm, whether they satisfy $p<s+n-i$. We call this number, {\it the number of steps $\Theta(n)$.} 

Before we state the precise result, we give a rough estimate. By \cite[Boundedness proposition 1]{ecker} (see also \cite[Lemma 6.3]{weingartner}), any combination \( (a_1, \dots, a_k) \) satisfying the inequality also fulfills the condition \( \prod_{i=1}^n a_i \leq 2n \). Therefore, the runs of our algorithm satisfy the same asymptotic upper bound, as the number of factorizations \( h(l) \) of an integer \( l \geq 1 \), where order is irrelevant and each factor is strictly greater than 1, satisfies the asymptotic estimate:
\begin{equation}\label{eq: opp}
\sum_{l \leq 2n} h(l) = O\left( 2n \cdot \frac{e^{2\sqrt{\log(2n)}}}{(\log(2n))^{3/4}} \right)
\end{equation}
(cf.~ \cite[Eq (4.51)]{oppenheim1927}.)

However, we have a more precise result, whose proof does not require an application of  Ecker's Boundedness proposition and also gives asymptotics for a lower bound:\footnote{In the theorem, we use $x(n)=O(a(n))$ in the following way: There exists constants $0<b<B$ such that $b\leq \lim_{n\rightarrow \infty} x(n)/a(n)\leq B$.} It rather counts the number of {\it all} solutions of \eqref{eq: sumprod1} for any $n'\leq n$:
\begin{theorem}\label{thm: complexity}
The number of steps \( \Theta(n) \) required to solve \eqref{eq: sumprod} is proportional to the number of solutions of \eqref{eq: sumprod1}, for any $n'\leq n$, and thus satisfies the asymptotics (from above and below)
\begin{equation}\label{eq: asy time}
\Theta(n) = O\left(\frac{n e^{2\sqrt{\log(n)}}}{\log^{3/4}(n)}\right).
\end{equation}
\end{theorem}

\begin{proof}
We begin by counting all steps in the algorithm where \( p' \leq s' + n - (i + 1) \) for the extended sequence \( a' = (a_1, \dots, a_i, a_{i+1}) \). When equality holds, we obtain a solution to~\eqref{eq: sumprod}. In the case of strict inequality, Lemma~\ref{lem: nyblom} implies that the sequence \( a' \) forms a solution for some \( i+1 \leq n' < n \). Moreover, by Theorem~\ref{th: eff working}, part~\ref{part: c}, such a solution for any \( n' \) does not repeat. Consequently, according to~\cite{weingartner}, the number of algorithmic steps satisfying  
\( p' \leq s' + n - (i + 1) \) follows the asymptotic behavior described in~\eqref{eq: asy time}.

The algorithm halts recursion at the first index \( i \) where  \( p' > s' + n - (i + 1) \). Since the number of such runs is at most twice the number of non-repeating solutions of length \( n' \leq n \), the overall time complexity likewise satisfies the asymptotic bound in~\eqref{eq: asy time}.
\end{proof}

If \( a = (a_1, \dots, a_n) \) is a solution, then by equation~\eqref{bta} in Section~\ref{sec: intro}, we have \( \prod_{i=1}^n a_i \leq 2n \). This bound does not explicitly appear in our algorithm, which raises the question of whether adding an extra pruning condition—namely, terminating a branch when \( \prod > 2n \)—would improve efficiency. However, this turns out not to be the case: with or without the additional condition, the same number of branches are pruned.

To see this, suppose the recursive function is called with arguments \( (s, p, i) \). Then
\[
p = a_1 a_2 \dots a_i < a_1 + \dots + a_i + (n - i) = s + (n - i).
\]
By~\cite[Lemma 6 (iii)]{weingartner}, if $p'\leq s'+n-(i+1)$, then the new product satisfies \( p'=p a_{i+1} \leq 2n \) as well. (This bound for product-sum inequalities generalizes the result in~\eqref{btb} due to Ecker that holds for product-sum equalities.) In other words, any branch not pruned by the original algorithm would also not be pruned by the additional condition, rendering the modification ineffective.

However, the number of function calls does not directly correspond to computation time. From a practical standpoint, checking the condition \( \prod_{k=1}^i a_k > 2n \) is computationally slightly cheaper than evaluating \( \prod_{k=1}^i a_k > \sum_{k=1}^i a_k + n - i \). Nonetheless, since the simpler condition is violated in only a small fraction of branches, the overall reduction in runtime is minimal. In fact, numerical experiments indicate that for \( n \approx 10^4 \), the additional pruning accounts for roughly 10\% of cases, while for \( n \approx 10^6 \), the impact drops to about 5\% and continues to decrease as \( n \) grows.

Table~\ref{tab: time comp} presents the number of algorithm steps (normalized as \emph{Runs} divided by the leading asymptotics in \eqref{eq: asy time}), the number of solutions for each case, and the algorithm's run time on an Intel i7-1365U processor. Results are shown using 8 and 12 threads to highlight the performance gains from parallel processing, particularly for large values of~$n$.
\begin{table}[h!]
\centering
\begin{tabular}{|c|r|r|r|r|r|}
\hline
\textbf{\( m \)} & \textbf{n} & \textbf{\# Runs / \( A(n) \)} & \textbf{\# Solutions} & \textbf{time - 8} & \textbf{time - 12} \\
\hline
3 & 7,294 & 0.0126 & 14 & 0.006 & -- \\
4 & 31,278 & 0.0099 & 14 & 0.008 & -- \\
5 & 124,158 & 0.0082 & 22 & 0.015 & -- \\
6 & 1,030,986 & 0.0064 & 43 & 0.09 & 0.07 \\
7 & 10,027,888 & 0.0051 & 50 & 1.25 & 0.99 \\
8 & 100,006,637 & 0.0042 & 137 & 37 & 32 \\
9 & 1,000,015,252 & 0.0036 & 150 & 357 & 271 \\
\hline
\end{tabular}
\caption{For values \( n = 10^m \), where \( 2 < m < 9 \), we report the number of algorithm runs (normalized as \emph{Runs} divided by the leading-order asymptotic expression 
\( A(n) = \frac{n e^{2\sqrt{\log(n)}}}{\log^{3/4}(n)} \)), along with the corresponding number of solutions to equation~\eqref{eq: sumprod}. 
The final two columns show the system time (in seconds) required to execute the algorithm on an Intel i7-1365U processor, comparing performance with 8 and 12 threads. 
Note that computing times may vary depending on core utilization and thread scheduling.}
\label{tab: time comp}
\end{table}

\section{Fundamental Inequalities}\label{sec: inequalities}
In this section, we present some fundamental statements concerning integer sequences whose product equals their sum.\footnote{The problem of this paper is old and elementary, hence similar statements may be found in the literature, though not formulated in the same manner.}
\begin{theorem}\label{prop: 1}
Let $n\geq 2$. Assume that the sequence $a=(a_1,a_2,\dots,a_n)$ with \(a_1 \geq a_2 \geq \cdots \geq a_n\) satisfies equation \eqref{eq: sumprod}. Then, for any \(1 \leq k \leq n\), 
\begin{equation}\label{inequality}
\prod_{i=1}^k a_i \leq \sum_{i=1}^k a_i + n - k.
\end{equation}
Let $k_\star$ be the greatest integer $k$ for which $a_k>1$. Then \eqref{inequality} holds as equality for $k\geq k_\star$,
and as strict inequality for $k<k_\star$. 
\end{theorem}
\begin{proof}
$n\geq 2$ implies that $k_\star\geq 2$ (that is, any solution has at least two integer factors that are strictly greater than one). The first statement can be easily verified for $n=2$, where the only solution of \eqref{eq: sumprod} is $a=(2,2)$. Therefore, we assume $n\geq 3$ for the rest of the statement.

By Definition of $k_\star$, $a_{k\star+1}=a_{k_\star+2}=\dots=a_n=1$. Therefore, \eqref{inequality} holds for $k\geq k_\star$ with equality. Next we show that for $k<k_\star$, the inequality is strict. For $k=1$ this is obviously true.

We proceed by induction, however by decreasing the index $k$ (and thus, $k=1$ is not the basis): To establish that the statement holds for the inductive basis $k=k_\star-1$, we compute
\begin{equation}\label{eq: enigmatic ineq}
a_{k_\star}\prod_{i=1}^k a_i= \prod_{i=1}^{k_\star} a_i =\sum_{i=1}^{k_\star} a_i + n - k_\star= \sum_{i=1}^k a_i+a_{k_\star}+(n-k)-1
\end{equation}
and thus division by $a_{k_\star}$ we get\footnote{Note that $s_k:=\sum_{i=1}^k a_i>1$ and $a_{k_\star}\geq 2$, hence also $s_k(a_{k_\star}-1)>a_{k_\star}-1$, whence $s_k+a_{k_\star}-1<s_k a_{k_\star}$ and substituting this inequality into \eqref{eq: enigmatic ineq} yields \eqref{eq: ind step}.}
\begin{equation}\label{eq: ind step}
\prod_{i=1}^k a_i <\sum_{i=1}^k a_i+(n-k).
\end{equation}
To make the inductive step, assume \eqref{eq: ind step} holds for some $1<k<k_\star$. Then
$$
a_k \prod_{i=1}^{k-1}a_i<\sum_{i=1}^k a_i+(n-k)=\sum_{i=1}^{k-1}a_i+a_k+(n-k),
$$
and division by $a_k$ gives
$$
\prod_{i=1}^{k-1}a_i<\sum_{i=1}^{k-1}a_i+1+(n-k)=\sum_{i=1}^{k-1}a_i+(n-(k-1)).
$$
\end{proof}

The following two auxilliary statements, also rely on the fundamental inequality of Theorem \ref{prop: 1}. The following statement in particular implies that by modifying the $1$'s in a solution of \eqref{eq: sumprod} no further solutions can be obtained.
\begin{lemma}\label{lem: 2}
Suppose $n>2$ and $k<n$. If $(a_1,\dots,a_k,1,\dots,1)$ satisfies $\prod_{i=1}^k a_i\geq\sum_{i=1}^k a_i+(n-k)$, then the extended sequence 
$$a'=(a_1,\dots,a_k,a_{k+1}',1,1,\dots,1)
$$
with $a_{k+1}'>1$ satisfies $\prod_{i=1}^{k+1} a_i'>\sum_{i=1}^{k+1} a_i'+(n-k-1)$ and thus does not solve \eqref{eq: sumprod}.
\end{lemma}
\begin{proof}
Note that $a_1<a_1+n-1$, hence the assumed inequality $\leq$ implies that $k\geq 2$. Since $a_{k+1}'\geq 2$,
\begin{align*}
\prod_{i=1}^{k+1}a_i'&=a_{k+1}'\prod_{i=1}^k a_i\geq 2 \prod_{i=1}^k a_i\geq 2\sum_{i=1}^k a_i+2(n-k)\\&> \sum_{i=1}^{k} a_i+a_k+2(n-k)\geq\sum_{i=1}^{k+1} a_i'+(n-(k-1))
\end{align*}
where the strict inequality holds because $k\geq 2$, and the last inequality holds, because $k<n$ and $a_{k+1}'\leq a_k$. Thus $a'$ does not satisfy the inequality \eqref{inequality}. Hence $a'$ does not solve \eqref{eq: sumprod}.
\end{proof}

A simpler modification is increasing solutions (in the sense of lexicographic order), without altering the length $k_\star$ of the non-trivial part of the solution; also this modification does not yield any solution:
\begin{lemma}\label{lem: 2x}
Suppose $n>2$ and $k<n$. If $(a_1,\dots,a_k,1,\dots,1)$ satisfies $\prod_{i=1}^k a_i\geq\sum_{i=1}^k a_i+(n-k)$, then the modified sequence 
$$a'=(a_1',\dots,a_k',1,\dots,1)
$$
with $a_{j}'\geq a_j$ for any $1\leq j\leq k$ such that $a'\neq a$, satisfies  $\prod_{i=1}^k a_i'>\sum_{i=1}^k a_i'+(n-k)$. In particular, $a'$ does not solve \eqref{eq: sumprod}.
\end{lemma}
\begin{proof}
By an inductive argument, it suffices to show that an increase of one $a_j$, $1\leq j<k$ to $b=(a_1,\dots,a_{j-1}, (a_j+1),a_{j+1},\dots, a_k)$ (or, if $j=k$ to $b=(a_1,\dots,a_{k-1}, (a_k+1)$) yields a strict inequality. First, since $\prod_{i=1}^k a_i\geq\sum_{i=1}^k a_i+(n-k)$, at least one $a_i$, $1\leq i\leq k$ ($i\neq j$) must be strictly larger than $1$, and thus
$\frac{1}{a_j}\prod_{i=1}^k a_i=\prod_{i\neq j}a_i>1$. Hence,
$$
\prod_{i=1}^k b_i=\prod_{i=1}^k a_i+\frac{1}{a_j}\prod_{i=1}^k a_i>\sum_{i=1}^k a_i+(n-k)+1=\sum_{i=1}^k b_i+(n-k).
$$
Thus $a'$ does not satisfy the inequality \eqref{inequality}. Therefore, $a'$ cannot be a solution of \eqref{eq: sumprod}.
\end{proof}

The third, important fact limits the range for the solutions' summands $a_i$,  (cf. the \cite[Boundedness Proposition 1 and 2]{ecker}). Here is an independent proof that uses Theorem \ref{prop: 1}.
\begin{lemma}\label{lem: 1}
Any solution $(a_1,\dots,a_n)$ of \eqref{eq: sumprod} satisfies $a_i\leq n$. Any solution, except the basic one ($a_1=n, a_2=2$) satisfies $a_i<n$ (and hence $a_1<n$).
\end{lemma}
\begin{proof}
Since $a_i$ is a non-increasing sequence, it suffices to show that $a_1\leq n$. Suppose, for a contradiction, that $a_1>n$, say $a_1=n+r$ for some $r\geq 1$. By Theorem \ref{prop: 1}, we have for $k=2$ the inequality
\[
(n+r)a_2\leq n+r+a_2+(n-2)
\]
implying
\[
0\leq n(a_2-2)\leq -(a_2-1)(r-1)-1<0,
\]
an impossibility, whence indeed $a_i\leq n$.

Clearly, the basic solution satisfies $a_1=n$. Assume now $a_1=n$ and $k_\star=2$, then $na_2=n+a_2+n-2$, whence $a_2=2$, which means, $a$ is the basic solution. 

Assume now $k_\star\geq 3$, and for some $i\geq 1$, $a_i=n$. Then $a_1=n$. If $a_2=2$, then the solution $a$ would be obtained from the basic solution by appending an integer $a_3\geq 2$ (and $k_\star-3$ further integers greater than one), which is in contradiction with Lemma \ref{lem: 2}.  Thus $a_2>2$ and by Theorem \ref{prop: 1}, we get the strict inequality for $k=2$, that is,
\[
n a_2<n+a_2+n-2
\]
which implies $n(a_2-2)<a_2-2$, hence $n<1$. Thus $a_1<n$, as claimed.
\end{proof}

The final statement -- that is however not used in verifing the validity of the algorithm -- demonstrates that solutions of \eqref{eq: sumprod} have very few coefficients strictly larger than one. 
\begin{lemma}\label{lem: log length}
Let $n\geq 3$ and let $a$ satisfy the inequality
\begin{equation}\label{ineq: sumprod}
\prod_{i=1}^n a_i\leq\sum_{i=1}^n a_i.
\end{equation}
For $2\leq m\leq n$, let $k_{\star,m}$ be the greatest integer $k$ for which $a_k=m$. Then
\begin{equation}\label{eq: log}
k_{\star,m}< \log_m(n)+1.
\end{equation}
\end{lemma}
\begin{remark}
This statement pertains to the efficiency of our algorithm. For \( m = 2 \), \( k_{\star,2} \) denotes the exact number of coefficients \( a_i \) in solutions \( a = (a_1, a_2, \dots, a_n) \) of~\eqref{eq: sumprod} that are strictly greater than 1. Since the algorithm terminates as soon as the product exceeds the sum, it never explores sequences containing more than \( \lceil k_{\star,2} \rceil + 1 \) such elements.
\end{remark}
\begin{proof}[Proof of Lemma \ref{lem: log length}.]
First, we claim that for $n\geq 3$, $k_{\star,2}<n$ (and thus $k_{\star,m}<n$ for any $m>1$). In fact, this can be directly checked to hold for $3\leq n\leq 4$. Furthermore, for $n\geq 5$, we know by induction that $2^n>n^2$. Assume, for a contradiction, $k_{\star,2}=n$. Then $2^n\leq \prod_{i=1}^n a_i\leq\sum_{i=1}^na_i\leq n^2$, which is impossible. 

Thus we have to prove that if the product of $n \geq 3$ natural numbers $a_1 \geq a_2 \geq \cdots \geq a_n \geq 1$ satisfies \eqref{ineq: sumprod} and precisely $1\leq k<n$ of them are greater than or equal to $m$:
    \begin{gather}
        a_1 \geq a_2 \geq \cdots \geq a_k \geq m \label{geq_m} \\
        m > a_{k+1} \geq a_{k+2} \geq \cdots \geq a_n \label{less_than_m}
    \end{gather}
    Then:
    \[ k \leq \ceil {\log_m n}. \]

Proof: Since $a_1 \geq a_2 \geq \cdots \geq a_k$, we have:
    \begin{equation}\label{ineq: sumprodxx} 
        a_1 \ a_k^{k-1} \leq \prod_{i=1}^k a_i \leq \prod_{i=1}^n a_i\leq  \sum_{i=1}^n a_i
    \end{equation}
    From \eqref{geq_m} and \eqref{less_than_m}, we have:
    \begin{equation}\label{not_strict}
        \sum_{i=1}^n a_i < a_1 \ k +  m \ (n-k).
    \end{equation}
    Substituting this into \eqref{ineq: sumprodxx} we get:
    \[
        a_1 \ a_k^{k-1} < a_1 \ k + m \ (n-k).
    \]
    Dividing both sides by $a_1$, we get:
    \[
        a_k^{k-1} < k + \frac{m}{a_1} \ (n - k)\leq n,
    \]
the last inequality holding due to $a_1 \geq m$. Thus,
    \[
        k-1 < \log_{a_k} n
    \]
    From \eqref{geq_m}, we know that $a_k \geq m$, thus    $k < \log_m n + 1$.
\end{proof}


%
%
%
%
%
%
%
%
%
%

\vskip 20pt\noindent {\bf Acknowledgement.} 
We thank Prof. Bernd Kreussler (Mary Immaculate College, Limerick) for valuable feedback on this work.

\end{document}